\documentclass[a4paper]{amsart}
\usepackage{amsmath,amssymb,amsthm,amscd,mathtools}
\usepackage{enumerate,bm,color,here,url,hyperref}

\usepackage{graphicx}
\graphicspath{{./figure/}}

\newcommand{\fig}[3][width=12cm]{
\begin{figure}[htb]
 \centering 
 \includegraphics[#1,clip]{#2} 
 \caption{#3} 
\label{fig:#2}
\end{figure}}

\theoremstyle{plain}
	\newtheorem{thm}{Theorem}[section]
	\newtheorem{lem}[thm]{Lemma}
	\newtheorem{prop}[thm]{Proposition}
	\newtheorem{cor}[thm]{Corollary}
	
	\newtheorem{conj}[thm]{Conjecture}
	\newtheorem{ques}[thm]{Question}
	
	\newtheorem*{main1}{Theorem~\ref{thm:isotopy}}
	\newtheorem*{main2}{Theorem~\ref{thm:hyp_weave}}
\theoremstyle{definition}
	
\theoremstyle{remark}

\DeclareMathOperator{\Homeo}{Homeo}
\DeclareMathOperator{\inter}{int}
\DeclareMathOperator{\Isom}{Isom}

\DeclareMathOperator{\vol}{vol}

\newcommand{\bbZ}{\mathbb{Z}}
\newcommand{\bbR}{\mathbb{R}}

\begin{document}

\title{On Isotopies and hyperbolicity of weaves}

\author{Ken'ichi Yoshida}
\address{International Institute for Sustainability with Knotted Chiral Meta Matter (WPI-SKCM$^2$), Hiroshima University, 1-3-1 Kagamiyama, Higashi-Hiroshima, Hiroshima 739-8531, Japan}
\email{kncysd@hiroshima-u.ac.jp}
\subjclass[2020]{57K10, 57K32, 57Q37}
\keywords{Links in the thickened torus, Hyperbolic links}
\date{}

\begin{abstract}
A weave is a type of textile that consists of vertical and horizontal threads, 
and typically it has a periodic structure. 
In this paper, we regard a weave as a link in the thickened torus with a diagram consisting of closed geodesics. 
As main results, we characterize isotopies and hyperbolicity of weaves to determine them from diagrams. 
Moreover, we show that there does not exist an essential Conway sphere for a weave. 
We use normal positions of essential surfaces of weave complements to describe them. 
\end{abstract}

\maketitle

\section{Introduction}
\label{section:intro}

A weave is a type of textile produced by combining threads (called warps and wefts) in two directions. 
The three types of weaves in Figure~\ref{fig:basic_weave.pdf} are commonly used as basic weaves. 
The plain weave is represented by an alternating diagram and has the simplest structure. 
A twill weave has a pattern of diagonal ribs. 
A satin weave has a structure in which most of the wefts (or the warps) are in front. 
Twill and satin weaves have variations by parameters, 
and their structures were mathematically investigated by Gr\"{u}nbaum and Shephard \cite{GS80}. 
Recently, there have been various studies to create weaves from molecules \cite{JS21, LOTY18, ZAALZ22}.

\fig[width=12cm]{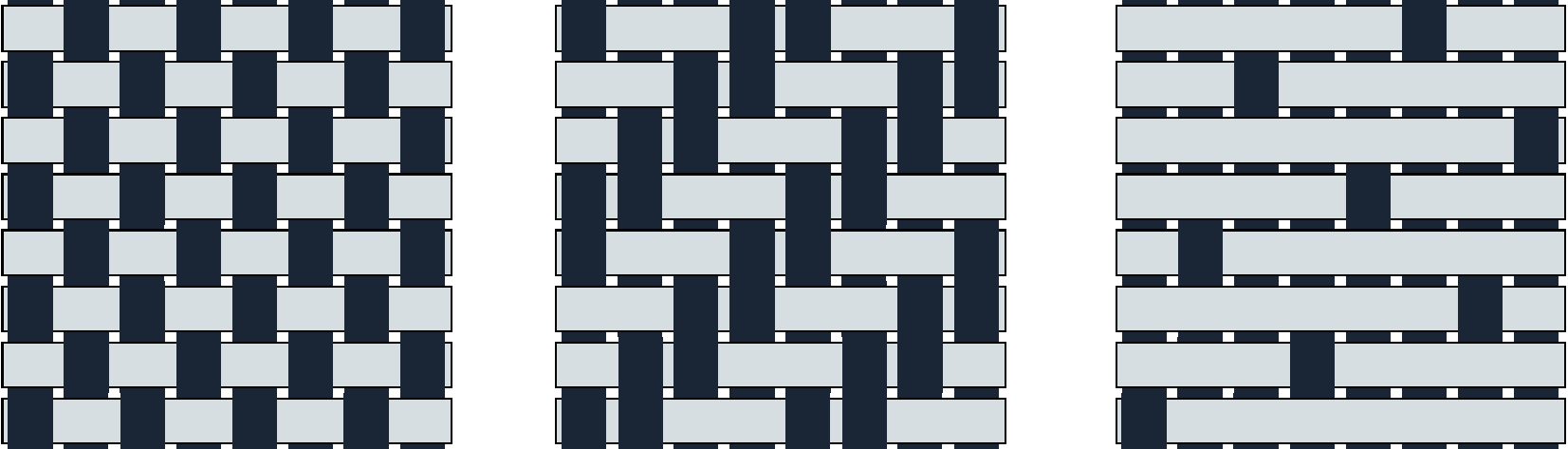}{Basic weaves: plain, twill, and satin}

Mathematical research on textile structures, including weaves, using knot theory was initiated by Grishanov, Meshkov, and Omel'chenko \cite{GMO07}. 
A typical textile structure can be represented by the preimage of a link in the thickened torus $T^{2} \times I$ by the universal covering map, which is called a doubly periodic tangle \cite{DLM26}. 
To consider doubly periodic tangles, several invariants such as polynomial invariants of links in $T^{2} \times I$ have been computed. 
However, a doubly periodic tangle is represented by infinitely many links in $T^{2} \times I$. 
These are related by finite covers. 
Nonetheless, consistency of finite covers and isotopies of links in $T^{2} \times I$ was shown in \cite{KMMY25} using hyperbolic geometry. 
Hyperbolic structures on the complements of weaves (especially the plain weave \cite{CKP16, Yoshida22}) have been investigated as interesting examples. 
More specifically, equivalence classes of weaves in which the warps and wefts have the same number of over- and under-crossings were classified in \cite{FKM23}. 
Stable configurations of weaves with repulsive interactions were investigated in \cite{KNSS26}.

In this paper, we regard a weave as a link in the thickened torus $T^{2} \times I$ with a diagram on $T^{2}$ that consists of closed geodesics. 
A weave can be regarded as a doubly periodic analogue of a braid closure in the solid torus. 
A rod packing in the 3-torus, which consists of closed geodesics, investigated in \cite{Hui25, HP24}, can be regarded as a triply periodic analogue. 

In the following main theorems, we characterize isotopies and hyperbolicity of weaves. 

\begin{main1}
If two weaves $L_{0}$ and $L_{1}$ in $T^{2} \times I$ are isotopic links, 
then $L_{0}$ is isotopic to $L_{1}$ via weaves. 
In other words, there is an ambient isotopy from $L_{0}$ to $L_{1}$ in which each intermediate link is a weave. 
\end{main1}

\begin{main2}
An $m \times n$-weave $L$ in $T^{2} \times I$ for $m, n \geq 1$ is hyperbolic if and only if 
\begin{itemize}
\item $L$ is not layered, and 
\item $L$ has no pair of parallel components even after interchanging adjacent comparable components. 
\end{itemize}
\end{main2}

Here, a weave is called layered if it is separated into two weaves. 
Consequently, we can determine from diagrams whether two weaves are isotopic and whether a weave is hyperbolic. 
For example, the three $8 \times 8$-weaves in Figure~\ref{fig:basic_weave.pdf} are hyperbolic and not mutually isotopic. 
In general cases, we have to be careful with the fact that two adjacent components in a certain condition can be interchanged.

In Section~\ref{section:isotopy}, we prepare notions and show Theorem~\ref{thm:isotopy}. 
In Section~\ref{section:torus}, we describe the essential tori in a weave complement to show Theorem~\ref{thm:hyp_weave}. 
We also discuss how generic hyperbolic weaves are. 
In Section~\ref{section:conway}, we show that there does not exist an essential Conway sphere for a weave. 
As a result, if $L$ is a hyperbolic weave, the thickened torus $T^{2} \times I$ admits a hyperbolic orbifold structure with singular locus $L$ of a cone angle $\pi$. 
In Section~\ref{section:question}, we ask questions on symmetries and hyperbolic volumes of weaves. 
We also give remarkable examples of weaves.

\section{Isotopies of weaves}
\label{section:isotopy}

In this section, we prepare notions and show a fundamental result on isotopies of weaves. 
A \emph{link} in a 3-manifold $X$ is a disjoint union of circles embedded in the interior of $X$. 
For two links $L_{0}$ and $L_{1}$ in $X$, 
an \emph{ambient isotopy} from $L_{0}$ to $L_{1}$ is a continuous map $F \colon X \times [0,1] \to X$ 
such that $F_{t} = F(\cdot, t) \colon X \to X$ is a homeomorphism for each $0 \leq t \leq 1$, $F_{0}$ is the identity map, and $F_{1}(L_{0}) = L_{1}$. 
Two links $L_{0}$ and $L_{1}$ are \emph{isotopic} if there is an ambient isotopy from $L_{0}$ to $L_{1}$. 

Let $S^{1} = \bbR / \bbZ$, $T^{2} = \bbR^{2} / \bbZ^{2} = S^{1} \times S^{1}$, and $I = [0,1]$. 
Let $p_{1}, p_{2} \colon T^{2} \times I \to S^{1}$ denote the projections to the first and second coordinates, respectively. 
For $m, n \geq 0$, an \emph{$m \times n$-weave} is a link in the thickened torus $T^{2} \times I$ that is a disjoint union of $m$ warps and $n$ wefts, where 
\begin{itemize}
\item a \emph{warp} (vertical component) is the image of an embedding $\iota \colon S^{1} \to T^{2} \times I$ such that $p_{1} \circ \iota \colon S^{1} \to S^{1}$ is constant, and $p_{2} \circ \iota \colon S^{1} \to S^{1}$ is a homeomorphism, and 
\item a \emph{weft} (horizontal component) is the image of an embedding $\iota \colon S^{1} \to T^{2} \times I$ such that $p_{2} \circ \iota \colon S^{1} \to S^{1}$ is constant, and $p_{1} \circ \iota \colon S^{1} \to S^{1}$ is a homeomorphism. 
\end{itemize}
We often call it a weave by ignoring $m$ and $n$. 
Note that a link isotopic to a weave may not be a weave in our terminology. 

A \emph{weaving diagram} of an $m \times n$-weave is the union of circles 
$\{ (2i-1)/2m \} \times S^{1}$ and $S^{1} \times \{ (2j-1)/2n \}$ for $i = 1, \dots, m$ and $j = 1, \dots, n$ on the torus $T^{2} = I \times I / \sim$ with data of the overpass or underpass at each crossing. 
Since an $m \times n$-weave can be isotoped so that the warps and the wefts are respectively projected to mutually disjoint circles, its isotopy class is represented by a weave diagram. 
A weaving diagram corresponds to a function $c \colon \{ 1, \dots, m \} \times \{ 1, \dots, n \} \to \{ 0,1 \}$ 
so that for the crossing of the $i$-th warp and the $j$-th weft 
\[
c(i,j) = 
\begin{cases}
0 & \text{if the warp is below,} \\
1 & \text{if the warp is above.} 
\end{cases}
\]
A weaving diagram is often drawn as a tiling of white and black squares as in \cite{GS80}. 
Some figures of weaves in this paper are intended to resemble such tilings. 

There are the following two relations between weaving diagrams. 
First, a translation on the torus $T^{2}$ maps a weaving diagram to another one. 
It corresponds to an element of $\bbZ / m\bbZ \times \bbZ / n\bbZ$ acting on $\{ 1, \dots, m \} \times \{ 1, \dots, n \}$. 
Next, the $i_{0}$-th and $i_{1}$-th warps are \emph{comparable} with respect to the wefts 
if the two functions $c(i_{0}, \cdot ), c(i_{1}, \cdot) \colon \{ 1, \dots, n \} \to \{ 0,1 \}$ 
(called the \emph{crossing function} of the warps) are comparable, 
i.e., $c(i_{0}, j) \leq c(i_{1}, j)$ for any $j \in \{ 1, \dots, n \}$, or $c(i_{0}, j) \geq c(i_{1}, j)$ for any $j \in \{ 1, \dots, n \}$. 
We define the comparability of two wefts in the same way. 
We can \emph{interchange} the two adjacent comparable components as shown in Figure~\ref{fig:interchange.pdf}. 
These two relations preserve the isotopy class of a weave. 
Note that two interchanges of two warps and two wefts are commutative. 

\fig[width=12cm]{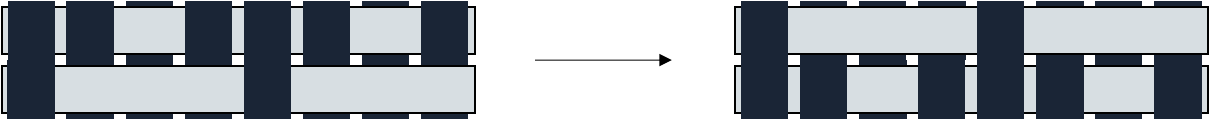}{An interchange of two wefts}

We show that these two relations are sufficient. 
Since a generic isotopy via weaves is represented by a finite sequence of isotopies on $T^{2}$ and interchanges of adjacent components, 
Theorem~\ref{thm:isotopy} implies Corollary~\ref{cor:isotopy}. 
Theorem~\ref{thm:isotopy} can be compared with the fact \cite[Theorem 2.1]{KT08} 
that if two braid closures are isotopic in the solid torus, 
then the corresponding elements of the braid group are conjugate, 
and so they are isotopic via braid closures. 

\begin{thm}
\label{thm:isotopy}
If two weaves $L_{0}$ and $L_{1}$ in $T^{2} \times I$ are isotopic links, 
then $L_{0}$ is isotopic to $L_{1}$ via weaves. 
In other words, there is an ambient isotopy from $L_{0}$ to $L_{1}$ in which each intermediate link is a weave. 
\end{thm}

\begin{cor}
\label{cor:isotopy}
Two weaving diagrams represent isotopic links in $T^{2} \times I$ 
if and only if they are connected by a translation on $T^{2}$ and interchanges of adjacent comparable components. 
\end{cor}

In the following lemmas, we show that if two warps (or wefts) can be interchanged, they are comparable. 
Let $L_{0}$ and $L_{1}$ be $m \times n$-weaves such that the warps and the wefts are respectively projected to mutually disjoint circles. 
For $k=0,1$, let $K_{1k}, \dots, K_{mk}$ denote the warps of $L_{k}$, and let $K'_{1k}, \dots, K'_{nk}$ denote the wefts of $L_{k}$. 
Suppose that there is an ambient isotopy $F$ from $L_{0}$ to $L_{1}$ such that $F_{1}(K_{i0}) = K_{i1}$ for $1 \leq i \leq m$ and $F_{1}(K'_{j0}) = K'_{j1}$ for $1 \leq j \leq n$. 
Moreover, suppose that $K_{i0}$ and $K_{j0}$ are the $i$-th warp and the $j$-th weft, 
but $K_{i1}$ and $K_{j1}$ are not necessarily so. 
The crossing functions and comparability of warps and wefts are defined through the weaving diagram. 
We first show that the comparability of warps (and wefts) is preserved by an isotopy. 

\begin{lem}
\label{lem:cross}
Let $L_{k} = K_{1k} \cup \dots \cup K_{mk} \cup K'_{1k} \cup \dots \cup K'_{nk}$ for $k=0,1$ and $F$ be as above. 
Then the two warps $K_{i_{0}0}$ and $K_{i_{1}0}$ are comparable 
if and only if the two warps $K_{i_{0}1}$ and $K_{i_{1}1}$ are comparable. 
\end{lem}
\begin{proof}
For two oriented components $K$ and $K'$ of a link in $T^{2} \times I$, 
the linking number of $K$ over $K'$ is defined as the sum of the signs of the crossings of $K$ above $K'$ in a diagram on $T^{2}$. 
The linking number is invariant under the Reidemeister moves, which is defined as usual, 
and so under the isotopies of a link. 
Hence the warp $K_{i0}$ is above the weft $K'_{j0}$ 
if and only if the warp $K_{i1}$ is above the weft $K'_{j1}$. 

Suppose that $\sigma \colon \{ 1, \dots, n \}  \to \{ 1, \dots, n \}$ is the permutation such that $K'_{j1}$ is the $\sigma (j)$-th weft for each $j \in \{ 1, \dots, n \} $. 
Then the crossing function of $K_{i0}$ is the composite of the crossing function of $K_{i1}$ and $\sigma$. 
Hence comparability of warps is preserved. 
\end{proof}

To describe the interchangeability of warps precisely, 
we consider the preimages of weaves by an infinite cyclic covering map $\pi \colon \bbR \times S^{1} \times I \to T^{2} \times I$. 
We say that a component of the preimage of a weave is also a warp or a weft. 
We define the comparability of two warps in $\bbR \times S^{1} \times I$ with respect to the wefts in the same way.

\begin{lem}
\label{lem:square}
Let $\widetilde{W}'_{2} = \bbR \times \{ 1/4, 3/4 \} \times \{ 1/2 \} \subset \bbR \times S^{1} \times I$ for $i=1,2$. 
Suppose that $W_{0}$ and $W_{1}$ are the links in $\bbR \times S^{1} \times I \setminus \widetilde{W}'_{2}$ shown in the left of Figure~\ref{fig:non-comparable.pdf}. 
Then $W_{0}$ and $W_{1}$ are not isotopic in $\bbR \times S^{1} \times I \setminus \widetilde{W}'_{2}$. 
Consequently, if two warps in $\bbR \times S^{1} \times I \setminus \widetilde{W}'$ are not comparable with respect to the wefts $\widetilde{W}' = \bbR \times \{ n \text{ points} \} \times \{ 1/2 \} \subset \bbR \times S^{1} \times I$, 
they cannot be interchanged in $\bbR \times S^{1} \times I \setminus \widetilde{W}'$. 
\end{lem}
\begin{proof}
By attaching two 2-handles, 
we embed $\bbR \times S^{1} \times I \setminus \widetilde{W}'_{2}$ into the complement of the unknot $U$ in the 3-sphere $S^{3}$ as shown in the right of Figure~\ref{fig:non-comparable.pdf}. 
Then $W_{0} \cup U$ is the 3-component unlink, and $W_{1} \cup U$ is the Borromean rings. 
Hence $W_{0}$ and $W_{1}$ are not isotopic. 

If two warps in $\bbR \times S^{1} \times I \setminus \widetilde{W}'$ are not comparable, 
there is the configuration of $W_{0}$ or $W_{1}$. 
Hence these two warps cannot be interchanged in $\bbR \times S^{1} \times I \setminus \widetilde{W}'$. 
\end{proof}

\fig[width=12cm]{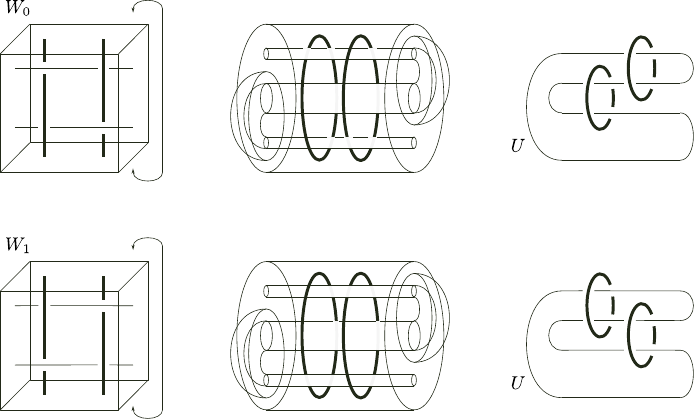}{The links $W_{0}$ and $W_{1}$ in $\bbR \times S^{1} \times I \setminus \widetilde{W}'_{2}$}

To arrange the warps and wefts in sequence, 
we show that if two warps are not comparable, 
they cannot be interchanged even by an ambient isotopy that does not fix the wefts. 

For a topological space $M$, let $\Homeo(M)$ denote the space of homeomorphisms on $M$ with the compact-open topology. 
For $S_{1}, S_{2} \subset M$, let $\Homeo(M, S_{1}; S_{2})$ denote the subspace of $\Homeo(M)$ consisting of the homeomorphisms that preserve $S_{1}$ setwise and fix $S_{2}$ pointwise. 
We omit $S_{i}$ if it is empty. 

An ambient isotopy $F$ on $T^{2} \times I$ fixing $\partial (T^{2} \times I)$ is lifted to a unique ambient isotopy $\widetilde{F}$ on $\bbR \times S^{1} \times I$ fixing $\partial (\bbR \times S^{1} \times I)$  such that $\pi \circ \widetilde{F}_{t} = F_{t} \circ \pi$ for each $0 \leq t \leq 1$. 
We recall that $p_{1} \colon T^{2} \times I \to S^{1}$ is the projection to the first coordinate. 
Let $\tilde{p}_{1} \colon \bbR \times S^{1} \times I \to \bbR$ also denote the projection to the first coordinate.

\begin{lem}
\label{lem:comparable}
Let $L_{k} = K_{1k} \cup \dots \cup K_{mk} \cup K'_{1k} \cup \dots \cup K'_{nk}$ for $k=0,1$ and $F$ be as above. 
For $1 \leq i \leq m$, let $\widetilde{K}_{i0}^{s} \subset \bbR \times S^{1} \times I \ (s \in \bbZ)$ denote the components of $\pi^{-1}(K_{i0})$. 
Let $\widetilde{K}_{i1}^{s} = \widetilde{F}_{1}(\widetilde{K}_{i0}^{s})$ by the lift $\widetilde{F}$ of $F$, which is a component of $\pi^{-1}(K_{i1})$. 
Let $a(i, s) = \tilde{p}_{1}(\widetilde{K}_{i0}^{s}), b(i, s) = \tilde{p}_{1}(\widetilde{K}_{i1}^{s}) \in \bbR$. 
Let $i_{0}, i_{1} \in \{ 1, \dots, m \}$ and $s_{0}, s_{1} \in \bbZ$. 
If $a(i_{0}, s_{0}) < a(i_{1}, s_{1})$ and $b(i_{0}, s_{0}) > b(i_{1}, s_{1})$, 
then the warps $K_{i_{0}0}$ and $K_{i_{1}0}$ are comparable. 
\end{lem}
\begin{proof}
By isotopy, 
we may assume that $W' = K'_{10} \cup \dots \cup K'_{n0} = K'_{11} \cup \dots \cup K'_{n1} = S^{1} \times P_{n}$, 
where $P_{n} \subset \inter(S^{1} \times I)$ consists of  $n$ distinct points. 
Note that $K'_{j0}$ and $K'_{j1}$ do not necessarily coincide. 
Moreover, we may assume that the ambient isotopy $F$ fixes the boundary $\partial (T^{2} \times I)$ pointwise, 
and $F_{1}|_{W'} \colon W' \to W'$ preserves the first coodinante of $T^{2} \times I$. 
Consequently, $F_{1} \in \Homeo(T^{2} \times I, W'; \partial (T^{2} \times I))$. 

Waldhausen's theorem \cite{Waldhausen68} implies that 
there is a natural isomorphism 
\[
\mathcal{CB}_{n} \times \bbZ \cong \pi_{0}(\Homeo(T^{2} \times I, W'; \partial (T^{2} \times I))), 
\] 
where $\mathcal{CB}_{n} = \pi_{0}(\Homeo(S^{1} \times I, P_{n}; \partial (S^{1} \times I)))$ is the $n$-strand annular (or circular) braid group, 
The inclusion $\mathcal{CB}_ {n} \hookrightarrow \pi_{0}(\Homeo(T^{2} \times I, W'; \partial (T^{2} \times I)))$ is induced by the natural inclusion 
\[
\begin{array}{ccc}
\Homeo(S^{1} \times I, P_{n}; \partial (S^{1} \times I)) & \longrightarrow & \Homeo(T^{2} \times I, W'; \partial (T^{2} \times I)) \\
\rotatebox{90}{$\in$} & & \rotatebox{90}{$\in$} \\
\hat{g} & \longmapsto & \mathrm{id}_{S^{1}} \times \hat{g}. 
\end{array}
\]
A generator of the $\bbZ$-factor is represented by a twist 
\[
\begin{array}{ccc}
T^{2} \times I & \longrightarrow & T^{2} \times I \\
\rotatebox{90}{$\in$} & & \rotatebox{90}{$\in$} \\
(x, y, z) & \longmapsto & (x+ z, y, z). 
\end{array}
\]
Since $F_{1}$ is isotopic to the identity of $T^{2} \times I$ fixing $\partial (T^{2} \times I)$, 
we may regard the mapping class $[F_{1}] \in \pi_{0}(\Homeo(T^{2} \times I, W'; \partial (T^{2} \times I)))$ as an element of $\mathcal{CB}_{n}$. 

Let $g = \mathrm{id}_{S^{1}} \times \hat{g} \in \Homeo(T^{2} \times I, W'; \partial (T^{2} \times I))$ such that $\hat{g} \in \Homeo(S^{1} \times I, P_{n}; \partial (S^{1} \times I))$ and $[g] = [F_{1}] \in \mathcal{CB}_{n}$. 
For each $1 \leq i \leq n$, 
the loop $g^{-1}(K_{i1}) = g^{-1} \circ F_{1}(K_{i0}) \subset (S^{1} \times I \setminus P_{n})_{i} = \{ (x,y,z) \in T^{2} \times I \setminus W' \mid x = p_{1}(K_{i1}) \}$ is isotopic (in particular, homotopic) to the warp $K_{i0}$ in $T^{2} \times I \setminus W'$, 
since $[g^{-1} \circ F_{1}] \in \mathcal{CB}_{n}$ is trivial. 
By projecting this homotopy to $(S^{1} \times I)_{i}$, 
we obtain a homotopy from the loop $g^{-1} \circ F_{1}(K_{i0})$ to a warp $K_{i2}$ (which is the projection of $K_{i0}$) in the $n$-punctured annulus $(S^{1} \times I \setminus P_{n})_{i}$. 
Due to Epstein \cite{Epstein66}, the simple closed curves $g^{-1} \circ F_{1}(K_{i0})$ and $K_{i2}$ are isotopic in $(S^{1} \times I \setminus P_{n})_{i}$. 
Hence there is $h \in \Homeo(T^{2} \times I; W' \cup \partial (T^{2} \times I))$ such that $h$ preserves the first coordinate of $T^{2} \times I$, $[h] \in \pi_{0}(\Homeo(T^{2} \times I; W' \cup \partial (T^{2} \times I)))$ is trivial, and $h(g^{-1} \circ F_{1}(K_{i0})) = K_{i2}$. 

Let $F'_{1} = h \circ g^{-1} \circ F_{1} \in \Homeo(T^{2} \times I; W' \cup \partial (T^{2} \times I))$. 
Since $[F'_{1}] \in \pi_{0}(\Homeo(T^{2} \times I; W' \cup \partial (T^{2} \times I)))$ is trivial, 
there is an ambient isotopy $F'$ to $F'_{1}$ that fixes the wefts $W'$ and the boundary $\partial (T^{2} \times I)$ pointwise. 
The ambient isotopy $F'$ is from $L_{0}$ to $K_{12} \cup \dots \cup K_{m2} \cup W'$. 
We take the lift $\widetilde{F}'$ so that it fixes the boundary. 
The homeomorphism $\widetilde{F}'_{1}$ is uniquely determined by $F'_{1}$. 
Since $F'_{1}$ is taken so that its projection to the first coordinate coincides with that of $F_{1}$, 
we have $\tilde{p}_{1}(\widetilde{F}'_{1}(\widetilde{K}_{i0}^{s})) = b(i, s)$. 
Since comparability of warps is preserved by Lemma~\ref{lem:cross}, 
we may assume that $F$ fixes the wefts $W'$ pointwise. 

If $a(i_{0}, s_{0}) < a(i_{1}, s_{1})$ and $b(i_{0}, s_{0}) > b(i_{1}, s_{1})$, 
then the warps $\widetilde{K}_{i_{0}0}^{s_{0}}$ and $\widetilde{K}_{i_{1}0}^{s_{1}}$ in $\bbR \times S^{1} \times I \setminus \pi^{-1}(W')$ can be interchanged by $\widetilde{F}$. 
By Lemma~\ref{lem:square}, $\widetilde{K}_{i_{0}0}^{s_{0}}$ and $\widetilde{K}_{i_{1}0}^{s_{1}}$ are comparable. 
Hence the warps $K_{i_{0}0}$ and $K_{i_{1}0}$ are also comparable. 
\end{proof}

\begin{proof}[Proof of Theroem~\ref{thm:isotopy}]
Let $L_{k} = K_{1k} \cup \dots \cup K_{mk} \cup K'_{1k} \cup \dots \cup K'_{nk}$ for $k=0,1$ and $F$ as above. 
Let $a(i, s), b(i, s) \in \bbR$ be as in Lemma~\ref{lem:comparable}. 
Suppose that $L_{1/2} = K_{1,1/2} \cup \dots \cup K_{m,1/2} \cup K'_{10} \cup \dots \cup K'_{n0}$ is an $m \times n$-weave such that $p_{1}(K_{i,1/2}) = [b(i, s)] \in S^{1}$, and the crossing function of $K_{i,1/2}$ coincides with that of $K_{i0}$ for each $1 \leq i \leq m$. 

We construct an isotopy from $L_{0}$ to $L_{1/2}$ via weaves $L_{t} = K_{1t} \cup \dots \cup K_{mt} \cup K'_{10} \cup \dots \cup K'_{n0}$ for $0 \leq t \leq 1/2$ that fixes the wefts. 
Let $e(i, s, t) = (1-2t) a(i, s) + 2t b(i, s)$ for $1 \leq i \leq m$, $s \in \bbZ$, and $0 \leq t \leq 1/2$. 
We consider the segments $\sigma (i, s) = \{ (e(i, s, t ), t) \in \bbR \times [0, 1/2] \}$. 
Two segments $\sigma (i_{0}, s_{0})$ and $\sigma (i_{1}, s_{1})$ with $a(i_{0}, s_{0}) < a(i_{1}, s_{1})$ intersect if and only if $b(i_{0}, s_{0}) > b(i_{1}, s_{1})$. 
By perturbing $a(i, s)$ and $b(i, s)$, we may assume that three of these segments do not intersect at a point. 

We take a continuous family $\widetilde{L}_{t} = (\bigsqcup_{1 \leq i \leq m, s \in \bbZ} \widetilde{K}_{it}^{s}) \sqcup \widetilde{W}'$ for $0 \leq t \leq 1/2$ such that 
\begin{itemize}
\item $\pi (\widetilde{L}_{0}) = L_{0}$ and $\pi (\widetilde{L}_{1/2}) = L_{1/2}$, 
\item each $\widetilde{K}_{it}^{s}$ is a warp, 
\item $\tilde{p}_{1} (\widetilde{K}_{it}^{s}) = e(i, s, t)$, and 
\item each $\widetilde{L}_{t}$ is preserved by the deck transformation of $\pi \colon \bbR \times S^{1} \times I \to T^{2} \times I$. 
\end{itemize}
If $e(i_{0}, s_{0}, t) = e(i_{1}, s_{1}, t) $, then the two segments $\sigma (i_{0}, s_{0})$ and $\sigma (i_{1}, s_{1})$ intersect. 
Since $K_{i_{0}0}$ and $K_{i_{1}0}$ are comparable by Lemma~\ref{lem:comparable}, 
the two warps $\widetilde{K}_{i_{0}t}^{s_{0}}$ and $\widetilde{K}_{i_{1}t}^{s_{1}}$ can be interchanged around the time $t$. 
Hence $L_{0}$ is isotopic to $L_{1/2}$ via weaves $L_{t} = \pi (\widetilde{L}_{t})$. 

In the same manner, $L_{1/2}$ is isotopic to $L_{1}$ via weaves. 
Hence $L_{0}$ is isotopic to $L_{1}$ via weaves. 
\end{proof}

\section{Essential tori in weave complements}
\label{section:torus}

In this section, we characterize the hyperbolic weaves. 
A link $L$ in $T^{2} \times I$ is \emph{hyperbolic} if the interior of $T^{2} \times I \setminus L$ admits a complete hyperbolic structure of finite volume. 
Since each component of a weave in $T^{2} \times I$ is homotopically non-trivial, 
a weave complement is irreducible (i.e., every embedded sphere bounds a ball). 
A compact orientable surface other than a sphere properly embedded in a compact orientable irreducible 3-manifold is \emph{essential} if it is incompressible, boundary-incompressible, and not boundary-parallel. 
Theorem~\ref{thm:hyperbolization} follows from Thurston's hyperbolization theorem for Haken 3-manifolds \cite{Thurston82}. 
The assumption excludes exceptionally simple Seifert fibered 3-manifolds such as the product of a 3-punctured sphere and a circle, in which there does not exist an essential torus, but there exists a $\pi_{1}$-injective immersed torus. 

\begin{thm}
\label{thm:hyperbolization}
Let $M$ be a compact orientable irreducible 3-manifold with boundary consisting of at least four tori. 
Then the interior $M$ admits a complete hyperbolic structure of finite volume 
if and only if there does not exist an essential torus in $M$. 
\end{thm}

A link $L$ in $T^{2} \times I$ is \emph{layered} if there is a torus $T$ (called a \emph{layering} torus) in $T^{2} \times I \setminus L$ such that $T$ is boundary-parallel in $T^{2} \times I$ but not boundary-parallel in $T^{2} \times I \setminus L$. 
Since a layering torus is essential in the complement, a layered link in $T^{2} \times I$ is not hyperbolic. 
The weave shown in the left of Figure~\ref{fig:nh_weave.pdf} is layered. 
For more details of layered weaves, we refer the reader to \cite{KNSS26}. 

Two adjacent components of a weave $L$ in $T^{2} \times I$ are \emph{parallel} if their crossing functions coincide. 
In this case, there exists an essential torus in $T^{2} \setminus L$ that surrounds the two components as shown in the right of Figure~\ref{fig:nh_weave.pdf}, and so $L$ is not hyperbolic. 
We show that it is sufficient to consider these two obstructions for hyperbolicity. 

\fig[width=12cm]{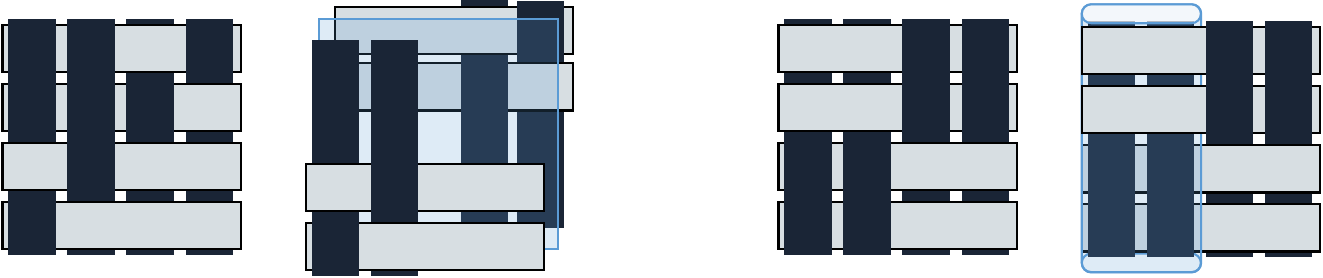}{Non-hyperbolic weaves}

\begin{thm}
\label{thm:hyp_weave}
An $m \times n$-weave $L$ in $T^{2} \times I$ for $m, n \geq 1$ is hyperbolic if and only if 
\begin{itemize}
\item $L$ is not layered, and 
\item $L$ has no pair of parallel components even after interchanging adjacent comparable components. 
\end{itemize}
\end{thm}

Note that every $m \times 1$-weave is layered, and so it is not hyperbolic. 

\begin{cor}
\label{cor:no_ac}
Let $L$ be an $m \times n$-weave in $T^{2} \times I$ for $m, n \geq 2$. 
Suppose that no pairs of adjacent components of $L$ are comparable. 
Then $L$ is hyperbolic. 
\end{cor}
\begin{proof}
Since no pairs of adjacent components are comparable, 
$L$ has no pair of parallel components even after interchanging adjacent comparable components. 
Assume that $L$ is layered. 
If $L$ is layered into an $m \times 0$-weave and a $0 \times n$ weave, 
then there are adjacent parallel components. 
Otherwise, there are two adjacent components in different layers, which are comparable. 
This contradicts the assumption. 
\end{proof}

Let $L$ be an $m \times n$-weave in $T^{2} \times I$ that is projected to 
\[
\left\{ \left( \frac{2i-1}{2m} ,y \right) \in T^{2} \ \middle| \ 1 \leq i \leq m \right\} \cup \left\{ \left( x, \frac{2j-1}{2n} \right) \in T^{2} \ \middle| \ 1 \leq j \leq n \right\}. 
\]
We decompose $T^{2} \times I$ into $4mn$ cubes along the faces 
\[
\left\{ \left( \frac{i}{2m}, y, z \right) \in T^{2} \times I \ \middle| \ 1 \leq i \leq 2m \right\} \cup \left\{ \left( x, \frac{j}{2n}, z \right) \in T^{2} \times I \ \middle| \ 1 \leq j \leq 2n \right\}, 
\]
whose union contains $L$ as shown in Figure~\ref{fig:weave_decomp.pdf}. 
The intersection of $L$ and each cube is a union of two arcs in adjacent faces. 
We use a normal position of an essential torus with respect to this decomposition. 

\fig[width=12cm]{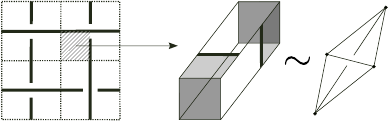}{A decomposition along surfaces containing a weave}

\begin{lem}
\label{lem:normal_closed}
Let $L$ be a weave in $T^{2} \times I$. 
Let $S$ be an essential closed surface in $T^{2} \times I \setminus L$. 
Then $S$ is isotopic in $T^{2} \times I \setminus L$ to a surface that is a union of disks 
each of which is one of the seven types shown in Figure~\ref{fig:normal1.pdf}. 
\end{lem}

\fig[width=12cm]{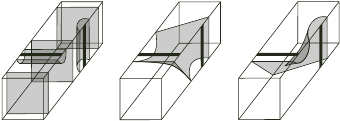}{Normal disks with respect to the decomposition for a weave}

The five disks in the left of Figure~\ref{fig:normal1.pdf} are rectangles, and the two disks in the center and right of Figure~\ref{fig:normal1.pdf} are hexagons. 
By collapsing the two squares in $\partial T^{2} \times I$ and the subsets of $L$ on the cube to four vertices, 
we obtain a tetrahedron as shown in the right of Figure~\ref{fig:weave_decomp.pdf}. 
Then the normal disks in Lemma~\ref{lem:normal_closed} are projected to the seven usual types of normal disks in a tetrahedron. 

\begin{proof}
We regard a cube in the decomposition as a polyhedron $P$ 
in which $L \cup P$ is a union of two edges. 
As a result, $P$ has 8 faces and 18 edges. 
Let $w(S)$ denote the number of intersectional points of $L$ and all the edges of the decomposition. 
We isotope $S$ in general position so that $w(S)$ is minimal. 

Suppose that the intersection of $S$ and a face contains a circle $C$. 
The circle $C$ bounds a disk $D_{0}$ in the face. 
Since $S$ is incompressible, the circle $C$ also bounds a disk $D_{1}$ in $S$. 
Since $T^{2} \times I \setminus L$ is irreducible, the sphere $D_{0} \cup D_{1}$ bounds a ball $B$. 
Then we can remove $C$ from the intersection by isotoping $S$ through $B$. 
Hence we may assume that the intersection of $S$ and each face does not contain a circle. 

For a polyhedron $P$ as above, consider a component $C$ of $S \cap \partial P$. 
Assume that there is an edge $e$ of $P$ such that $C \cap e$ contains two points. 
Fix an orientation of $C$. 
Since the circle $C$ bounds a disk in $\partial P$, there are two points of $C \cap e$ of opposite signs. 
Such two points innermost in $e$ bound an arc $\alpha$ in $e$. 
There is a disk $D_{0}$ in $\partial P$ such that $\partial D_{0} = \alpha \cup \beta$ and $\beta \subset C$. 
Then we can decrease $w(S)$ by isotoping $S$ through a neighborhood of $D_{0}$, 
which contradicts the minimality of $w(S)$. 
Hence $C$ intersects each edge at most once, 
and two edges in $C$ are contained in different faces of $P$. 

Moreover, $C$ bounds a disk $D$ in $S \cap P$. 
Otherwise, we can decrease $w(S)$ by isotoping $S$, since $S$ is incompressible. 
Then the disk $D$ is one of the seven types shown in Figure~\ref{fig:normal1.pdf}. 
\end{proof}

\begin{lem}
\label{lem:torus}
Let $L$ be a weave in $T^{2} \times I$, 
and let $T$ be an essential torus in $T^{2} \times I \setminus L$. 
After an ambient isotopy preserving the projection of $L$ to $T^{2}$, 
we have $T = \{ (x,y,z) \in T^{2} \times I \mid (x,z) \in C \}$ 
or $\{ (x,y,z) \in T^{2} \times I \mid (y,z) \in C \}$, 
where $C \subset S^{1} \times I$ is a circle. 
\end{lem}
\begin{proof}
By Lemma~\ref{lem:normal_closed}, $T$ is isotopic in $T^{2} \times I \setminus L$ to a normal position with respect to the decomposition into cubes. 
For this decomposition of $T$ into normal disks, 
let $v$, $e$, $f_{4}$, and $f_{6}$ denote the numbers of vertices, edges, rectangles, and hexagons, respectively. 
The Euler characteristic of $T$ is computed as $v - e + f_{4} + f_{6} = 0$. 
Since each vertex has degree four, we have $4v = 2e$. 
By removing $e$, we have $f_{4} + f_{6} = v$. 
By counting the vertices, we have $4f_{4} + 6f_{6} = 4v$. 
Hence $f_{6} = 0$. 
In other words, there are no hexagons in $T$. 
Moreover, if there is a rectangle in $T$ that surrounds a warp, 
then there is not a rectangle in $T$ that surrounds a weft. 
Indeed, if there are two such rectangles, then $T$ would be disconnected. 

Suppose that there is not a rectangle in $T$ that surrounds a weft. 
Then after an ambient isotopy preserving the projection of $L$ to $T^{2}$, 
we have that $T$ is fibered by lines in the $y$-coordinate. 
In other words, $T = \{ (x,y,z) \in T^{2} \times I \mid (x,z) \in C \}$ for a circle $C \subset S^{1} \times I$. 
It is similar in the case that there is not a rectangle in $T$ that surrounds a warp. 
\end{proof}

\begin{proof}[Proof of Theorem~\ref{thm:hyp_weave}]
If $L$ is layered, then a layering torus is essential. 
If $L$ has parallel components, a torus that surrounds them is essential. 
In these cases, $L$ is not hyperbolic. 

Conversely, suppose that $L$ is not hyperbolic. 
Then there is an essential torus $T$ in $T^{2} \times I \setminus L$ by Theorem~\ref{thm:hyperbolization}. 
We isotope $T$ as in Lemma~\ref{lem:torus}. 
If $C$ is homotopically non-trivial in $S^{1} \times I$, then $T$ is a layering torus. 
Otherwise, $T$ bounds a solid torus in $T^{2} \times I$ that contains parallel components of $L$. 
In the latter case, we can isotope $L$ via weaves so that there are adjacent parallel components. 
\end{proof}

We describe the JSJ decomposition of a weave complement. 
The JSJ decomposition \cite{JS79, Johannson79} of a compact orientable irreducible 3-manifold with boundary consisting of incompressible tori is a unique minimal decomposition along essential tori (called the \emph{JSJ tori}) into hyperbolic or Seifert fibered pieces (called the \emph{JSJ pieces}). 

\begin{cor}
\label{cor:jsj}
Let $L$ be a weave in $T^{2} \times I$. 
Each JSJ piece of the complement of $L$ is one of the following types: 
\begin{itemize}
\item a hyperbolic weave complement, 
\item an $m \times 0$ or $0 \times n$-weave complement, or 
\item the complement of links in a solid torus with at least two components parallel to the core. 
\end{itemize}
\end{cor}
\begin{proof}
If a JSJ torus bounds a solid torus in $T^{2} \times I$, this solid torus contains components parallel to the core, whose complement is Seifert fibered. 
If a weave in $T^{2} \times I$ has a Seifert fibered complement, 
then it is an $m \times 0$ or $0 \times n$-weave by \cite[Proposition 3.9]{KMMY25}. 
The remaining JSJ pieces are hyperbolic weave complements. 
\end{proof}

We also consider how generic hyperbolic weaves are. 
For $m, n \geq 1$, let $N_{\mathrm{hyp}}(m, n)$ denote the number of crossing functions 
$c \colon \{ 1, \dots, m \} \times \{ 1, \dots, n \} \to \{ 0,1 \}$ representing hyperbolic weaves. 
There are $2^{mn}$ crossing functions for fixed $m$ and $n$. 

\begin{prop}
\label{prop:proportion}
For fixed $m \geq 1$, we have 
\[
\lim_{n \to \infty} \frac{N_{\mathrm{hyp}}(m, n)}{2^{mn}} = 0 
\quad \text{and} \quad 
\lim_{n \to \infty} \frac{N_{\mathrm{hyp}}(n, n)}{2^{n^{2}}} = 1. 
\]
\end{prop}
\begin{proof}
If an $m \times n$-weave is hyperbolic, 
then the crossing function of each weft is not constant. 
Hence 
\[
N_{\mathrm{hyp}}(m, n) \leq (2^{m}-2)^{n}, 
\]
and so 
\[
\lim_{n \to \infty} \frac{N_{\mathrm{hyp}}(m, n)}{2^{mn}} 
\leq \lim_{n \to \infty} \left( 1- 2^{1-m} \right)^{n} = 0. 
\]

Corollary~\ref{cor:no_ac} implies that $N_{\mathrm{hyp}}(m, n) \geq N_{\mathrm{nc}}(m, n)$ for $m, n \geq 2$, 
where $N_{\mathrm{nc}}(m, n)$ is the number of crossing functions 
$c \colon \{ 1, \dots, m \} \times \{ 1, \dots, n \} \to \{ 0,1 \}$ representing weaves in which no pairs of adjacent components are comparable. 
For each $1 \leq i \leq m$, there are $3^{n}$ pairs of crossing functions $c(i, \cdot), c(i+1, \cdot) \colon \{ 1, \dots, n \} \to \{ 0,1 \}$ of the $i$-th and $i+1$-th warps such that $c(i, \cdot) \leq c(i+1, \cdot)$, 
which correspond to the functions $\{ 1, \dots, n \} \to \{ (0,0), (0,1), (1,1) \}$. 
Hence 
\[
N_{\mathrm{nc}}(m, n) \geq 2^{mn} - 2m \cdot 2^{(m-2)n} 3^{n} - 2n \cdot 2^{m(n-2)}3^{m}, 
\]
and so 
\[
\lim_{n \to \infty} \frac{N_{\mathrm{hyp}}(n, n)}{2^{n^{2}}} 
\geq \lim_{n \to \infty} \left( 1 - 4n \left( \frac{3}{4} \right)^{n} \right) = 1. 
\]
\end{proof}

Proposition~\ref{prop:proportion} can be compared with the results \cite{IM17, Ito15, Ma14} that generic links via random braid closure or random bridge position are hyperbolic. 
In contrast, Malyutin \cite{Malyutin19} showed that the proportion of hyperbolic links among the prime non-split links of at most $n$ crossings does not converge to 1 as $n \to \infty$ (see also \cite{BM19, Malyutin20}).

\section{Conway spheres for weaves}
\label{section:conway}

A \emph{Conway sphere} for a link $L$ in a 3-manifold $X$ is an embedded sphere in $X$ which intersects $L$ transversely at four points. 
We show that every Conway sphere for a weave is compressible in the complement. 

\begin{thm}
\label{thm:normal_conway}
Let $L$ be a weave in $T^{2} \times I$. 
Then there does not exist an essential Conway sphere for $L$ in $T^{2} \times I$. 
\end{thm}
\begin{proof}
Assume that $S$ is an essential Conway sphere. 
In the same manner as the proof of Lemma~\ref{lem:normal_closed}, 
we consider normal position of $S$ with respect to the cubical decomposition. 
Let $w(S)$ denote the number of intersectional points of $L$ and all the edges of the decomposition. 
We isotope $S$ in general position so that $w(S)$ is minimal. 
We may assume that the intersection of $S$ and each face does not contain a circle. 
Let $C$ be a component of the intersection of $S$ and the boundary of a polyhedron $P$. 
Then $C$ intersects each edge at most once, and $C$ bounds a disk $D$ in $S \cap P$. 
The difference from the proof of Lemma~\ref{lem:normal_closed} is that $C$ may intersect $L$. 
Moreover, two vertices of $C$ adjacent to $C \cap L$ cannot be contained in two adjacent edges of $P$. 
Otherwise, we can decrease $w(S)$ by isotoping $S$ in a neighborhood of a disk in $\partial P$ whose boundary is a union of four segments in $C$ and the adjacent edges. 
Then the disk $D$ is one of the 7 types in Figure~\ref{fig:normal1.pdf} and the 15 types in Figure~\ref{fig:normal2.pdf}. 
The type in the upper left of Figure~\ref{fig:normal2.pdf} is called a \emph{bigon}. 
The other 14 types can be regarded as modifications of the former 7 types intersecting $L$. 

\fig[width=12cm]{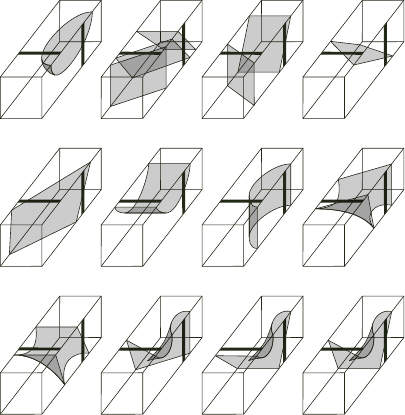}{Normal disks intersecting $L$}

For the decomposition of $S$ into normal disks, 
let $v_{4}$, $e$, and $f_{k,l}$ denote the numbers of vertices of degree four, edges, and faces with $k$ vertices contained in $L$ and the other $l$ vertices. 
Since $S$ intersects $L$ at four points, there are precisely four vertices of degree two. 
The Euler characteristic of $S$ is computed as $4 + v_{4} - e + \sum_{k,l} f_{k,l} = 2$. 
By considering the degrees of the vertices, we have $8 + 4v_{4} = 2e$. 
By removing $e$, we have $\sum_{k,l} f_{k,l} = v_{4} + 2$. 
By counting the vertices, we have $\sum_{k,l} k f_{k,l} = 8$ and $\sum_{k,l} l f_{k,l} = 4v_{4}$. 
By removing $v_{4}$, we have $\sum_{k,l} (4-l) f_{k,l} = 8$. 
Now $k = 0,1,2$ and $l =4, 6$ except for $(k,l) = (2,2)$. 
Hence $2f_{2,2} = \sum_{0 \leq k \leq 2} 2 f_{k,6} + 8 \geq 8$. 
By combining $\sum_{k,l} k f_{k,l} = 8$ and $f_{2,2} \geq 4$, 
we obtain $f_{2,2} = 4$ and $f_{k,l} = 0$ for $k \geq 1$. 
Hence the faces adjacent to a bigon are also bigons, 
and so $S$ is a union of four bigons. 
Then $S$ is compressible, which contradicts the assumption. 
\end{proof}

For $0 < \theta \leq \pi$, 
a link $L$ in a 3-manifold $X$ is called \emph{$\theta$-hyperbolic} 
if the interior of $X$ admits a hyperbolic cone-manifold structure of finite volume 
in which $L$ is the singular locus with a cone angle $\theta$. 
A cone-manifold with a cone angle $2\pi/n$ for $n \geq 2$ can be regarded as an orbifold. 
If $L$ is $\theta$-hyperbolic, then $L$ is hyperbolic. 
More strongly, they are joined with a continuous deformation of hyperbolic cone-manifold structures \cite{Kojima98}. 

As a special case, $\pi$-hyperbolicity of a knot in $S^{3}$ and decomposition of a knot in $S^{3}$ along essential Conway spheres were investigated in \cite{BS79}. 
More generally, a $\pi$-hyperbolic link in a 3-manifold can be characterized by Thurston's orbifold theorem, which was proved in \cite{BLP05, CHK00}. 
In the proof, hyperbolic cone-manifold structures are deformed from a hyperbolic structure by increasing cone angles. 
Deformation continues until an essential Euclidean cone-subsurface appears or the structures collapse. 
In our setting, the $\pi$-hyperbolic links are characterized as follows. 

\begin{thm}
\label{thm:pi-hyperbolization}
Let $L$ be a hyperbolic link in a compact orientable irreducible 3-manifold $X$ with boundary consisting of non-empty incompressible tori. 
Suppose that there does not exist an essential Conway sphere for $L$ in $X$. 
Then $L$ is $\pi$-hyperbolic. 
\end{thm}
\begin{proof}
Consider the orbifold $\mathcal{O}$ of the underlying space $X$ with a cone angle $\pi$ at the singular locus $L$. 
Thurston's orbifold theorem states that $\mathcal{O}$ is decomposed along (possibly empty) essential spherical and Euclidean 2-suborbifolds into pieces each of which is a hyperbolic, Seifert fibered, Euclidean, spherical, or Sol orbifold. 
Since the last two types are closed, we exclude them. 
A non-closed Euclidean 3-orbifold is Seifert fibered. 
The closed Euclidean 3-orbifolds that are not Seifert fibered were classified by Dunbar \cite{Dunbar88}. 

Since $X$ is irreducible, a sphere embedded in $X$ intersects $L$ transversely at an even number of points. 
In particular, there is not an embedded sphere with three intersectional points of $L$, called a turnover. 
Moreover, since $L$ is hyperbolic, there is not an essential sphere with two intersectional points of $L$. 
Hence there is not an essential spherical 2-suborbifold of $\mathcal{O}$. 
A possible essential Euclidean 2-suborbifold of $\mathcal{O}$ is a sphere with four cone points of cone angles $\pi$, called a pillowcase, which is a Conway sphere. 

Since there does not exist an essential Conway sphere for $L$, 
then $\mathcal{O}$ is hyperbolic or Seifert fibered. 
Assume that $\mathcal{O}$ is Seifert fibered. 
Then $\mathcal{O}$ is geometric (see \cite[Theorem 2.50]{CHK00}). 
By \cite[Theorem 1]{Dunbar88}, 
the 3-manifold $X$ is obtained by gluing (possibly empty) solid tori to a Seifert fibered manifold, 
and the link $L$ is a union of fibers and Montesinos links in the attached solid tori. 
Since $X$ has non-empty boundary, 
we have $L$ is a union of fibers of a Seifert fibered 3-manifold $X$, 
or there is an essential torus in $X \setminus L$. 
This contradicts the assumption that $L$ is hyperbolic. 
Hence $\mathcal{O}$ is hyperbolic. 
\end{proof}

\begin{cor}
\label{cor:pi-hyp_weave}
If a weave $L$ in $T^{2} \times I$ is hyperbolic, then $L$ is $\pi$-hyperbolic. 
\end{cor}

Note that hyperbolic cone-manifold structures on the alternating $2 \times 2$-weave were explicitly constructed in \cite{Yoshida22}.

\section{Further questions}
\label{section:question}

\subsection{Symmetries of weaves}

The preimage of a link in $T^{2} \times I$ by the universal covering map $\bbR^{2} \times I \to T^{2} \times I$ is called a \emph{doubly periodic tangle}. 
Actually, textile structures such as weaves should be regarded as doubly periodic tangles, rather than links in $T^{2} \times I$. 
A \emph{layer group} is a discrete group that acts $\bbR^{2} \times I$ cocompactly 
and corresponds to an orbifold that is a quotient of $T^{2} \times I$. 
The symmetry of a doubly periodic tangle is represented by one of the 80 layer groups, as explained in \cite{MDD25}. 
According to De Las Pe\~{n}as, Tomenes, and Liza \cite{PTL24}, 
there are precisely 50 layer groups that represent the symmetries of weaves. 
However, they classified the symmetries for fixed weaving diagrams 
and did not consider isotopy classes of weaves. 
Moreover, given examples of weaves with low symmetries often have adjacent parallel components. 
Nonetheless, the Mostow rigidity theorem \cite{Mostow73, Prasad73} implies that  
the symmetry of a hyperbolic weave $L$ in $T^{2} \times I$ up to isotopy is represented by 
the finite group $\Isom(T^{2} \times I \setminus L)$ consisting of the isometries on the hyperbolic 3-manifold $T^{2} \times I \setminus L$. 
For the doubly periodic tangle $\widetilde{L}$ in $\bbR^{2} \times I$ that is the preimage of $L$, 
the group $\Isom(\bbR^{2} \times I \setminus \widetilde{L})$ is a layer group. 

\begin{ques}
Can each of the 50 layer groups given in \cite{PTL24} be obtained from the symmetry of a hyperbolic weave up to isotopy? 
\end{ques}

\subsection{Hyperbolic volumes of weaves}

Let us consider the volume $\vol(L)$ of a $m \times n$-weave $L$, 
which is the sum of hyperbolic volumes of hyperbolic JSJ pieces of the complement of $L$. 
We have the volume function 
\[
\vol \colon \{ c \colon \{ 1, \dots, m \} \times \{ 1, \dots, n \} \to \{ 0,1 \} \} \to \bbR_{\geq 0} 
\] 
that assigns the volume of the corresponding weave to a crossing function.

As shown in \cite{CKP16}, the complement of an $m \times n$-weave $L$ is decomposed into $mn$ ideal octahedra. 
It is also obtained from the decomposition in Section~\ref{section:torus}: 
An ideal octahedron is obtained by gluing four ideal tetrahedra. 
Hence $\vol(L) \leq mn V_{\mathrm{oct}}$, where $V_{\mathrm{oct}} \approx 3.6638$ is the volume of a regular ideal octahedron. 
For even $m$ and $n$, the maximum is attained by the alternating weave.

For a hyperbolic weave $L$, let $\vol_{\pi}(L)$ denote the volume of the hyperbolic orbifold of underlying space $T^{2} \times I$ with a cone angle $\pi$ at singular locus $L$. 
For a general weave $L$, let $\vol_{\pi(L)}$ denote the sum of $\vol_{\pi}$ of the hyperbolic JSJ pieces of $T^{2} \times I \setminus L$. 
If $\vol(L) \neq 0$, the ratio $\vol_{\pi}(L) / \vol(L)$ is invariant under the finite covers of $T^{2} \times I$. 
Since the Schl\"{a}fli formula implies that the volumes decrease as the cone angles increase in deformation, 
we have $0 < \vol_{\pi}(L) / \vol(L) < 1$. 

For example, suppose that $L$ is the alternating $2 \times 2$-weave. 
Since the complement of $L$ can be decomposed into four regular ideal octahedra, 
we have $\vol(L) = 4V_{\mathrm{oct}}$. 
The hyperbolic orbifold to consider can be decomposed into four copies of a right-angled 4-trapezohedron (see \cite{Yoshida22} for details). 
This polyhedron can be decomposed into eight copies of a tetrahedron of dihedral angles $\pi/2$ and $\pi/4$ that is $1/16$ of a regular ideal octahedron as shown in Figure~\ref{fig:trapezohedron.pdf}. 
Hence $\vol_{\pi}(L) = 2V_{\mathrm{oct}}$, and so $\vol_{\pi}(L) / \vol(L) = 1/2$. 

\fig[width=9cm]{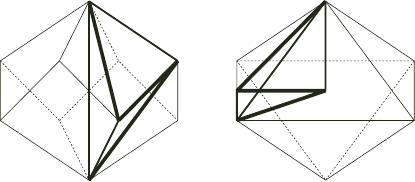}{Decompositions of a trapezohedron and an octahedron. The thick and thin edges of the two tetrahedra have dihedral angles $\pi/2$ and $\pi/4$ respectively.}

\begin{ques}
\
\begin{itemize}
\item Can we obtain an explicit expression of the volume function for weaves? 
\item What are the supremum and infimum of the ratios $\vol_{\pi}(L) / \vol(L)$ of the hyperbolic weaves $L$? 
\item How are the volumes of $m \times n$-weaves distributed for large $m$ and $n$? 
\item What is the hyperbolic $m \times n$-weave of the smallest volume? 
\end{itemize}
\end{ques}

We suggest the following conjecture on the last question. 

\begin{conj}
The weave shown in Figure~\ref{fig:small_weave.pdf} has the smallest volume among the hyperbolic $n \times n$-weaves for each $n \geq 2$. 
\end{conj}

\fig[width=3.5cm]{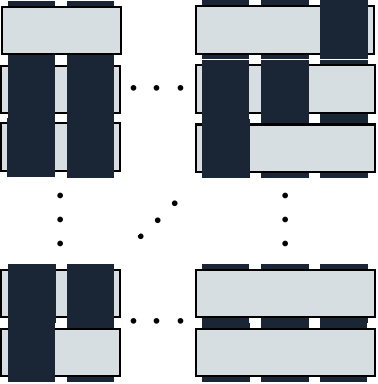}{An $n \times n$-weave with the conjecturally smallest volume}

The hyperbolic $2 \times 2$-weave is unique up to isotopy. 
There are precisely two hyperbolic $3 \times 3$-weaves up to homeomorphism as shown in Figure~\ref{fig:3-weave.pdf}, 
and their volumes are computed as approximately $24.0921$ and $26.7879$ by SnapPy \cite{SnapPy}. 

\fig[width=6cm]{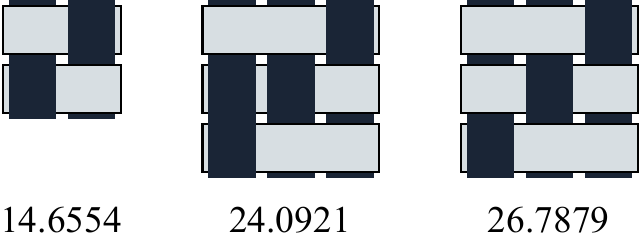}{The hyperbolic $2 \times 2$- and $3 \times 3$-weaves and their volumes}

The complement of the $n \times n$-weave in Figure~\ref{fig:small_weave.pdf} is homeomorphic to the complement of the minimally twisted $(2n+2)$-chain link. 
Indeed, each of the complements contains $2n+2$ 3-punctured spheres and two $(n+1)$-punctured spheres as shown in Figure~\ref{fig:chain.pdf}, 
and it is decomposed along these surfaces into four ideal $(n+1)$-antiprisms 
(the decomposition of the chain link was shown in \cite[Example 6.8.7]{Thurston78}). 
Both complements are obtained by the ``double of double'' (see \cite{Yoshida22}) along a checkerboard coloring of the faces of an antiprism, and so they coincide. 
This hyperbolic 3-manifold conjecturally has the smallest volume among the orientable hyperbolic 3-manifolds with $2n+2$ cusps for $2 \leq n \leq 4$. 
For $n \geq 5$, its volume is smaller than $(2n+2)V_{\mathrm{oct}}$, 
and $(2n+1)V_{\mathrm{oct}}$ is the smallest known volume of an orientable hyperbolic 3-manifold with $2n+2$ cusps as the author's knowledge (see \cite{KPR12} for details). 

\fig[width=12cm]{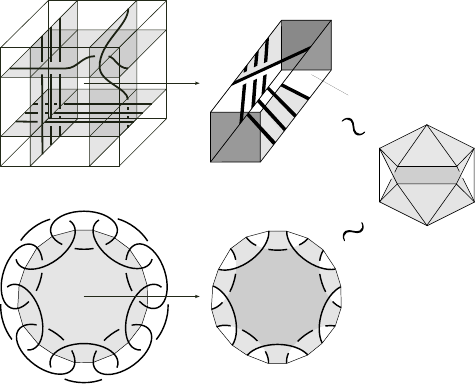}{Decompositions into four antiprisms}

Finally, we give an example that a weave is not uniquely determined by the complement up to homeomorphism. 
The complements of the two weaves in Figure~\ref{fig:weave_compl.pdf} are homeomorphic. 
In particular, they have the same volume, which is computed as approximately $34.4496$ by SnapPy \cite{SnapPy}. 
However, Corollary~\ref{cor:isotopy} implies that the two weaves are not mapped to each other by a homeomorphism on $T^{2} \times I$. 

\fig[width=12cm]{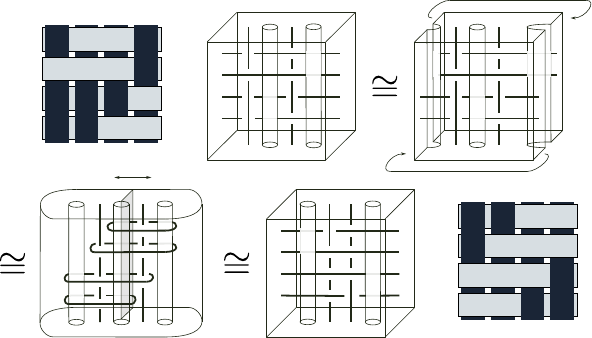}{Two weaves whose complements are homeomorphic}

\section*{Acknowledgements}
The author is grateful to Mizuki Fukuda, Katsuya Inoue, Kai Ishihara, Hiroki Kodama, Yuka Kotorii, Sonia Mahmoudi, Shunsuke Takano, and Wataru Yuasa for their helpful discussions. 
This work was supported by the World Premier International Research Center Initiative Program, International Institute for Sustainability with Knotted Chiral Meta Matter (WPI-SKCM$^2$), MEXT, Japan, 
and JSPS Program for Forming Japan's Peak Research Universities (J-PEAKS) Grant Number JPJS00420230011. 

\bibliographystyle{plain}
\bibliography{ref-ihw.bib}

\end{document}